\documentclass[12pt,reqno]{amsart}
\numberwithin{equation}{section}

\newtheorem{innercustomthm}{Theorem}
\newenvironment{customthm}[1]
  {\renewcommand\theinnercustomthm{#1}\innercustomthm\itshape}
  {\endinnercustomthm}

\usepackage{times}
\usepackage{enumerate}
\usepackage{graphicx,adjustbox}
\usepackage{hyperref}
\usepackage{amsmath,amssymb}
\usepackage{mathrsfs}
\usepackage{float}
\usepackage{caption}
\usepackage{subcaption}
\usepackage[margin=1.25 in]{geometry}  
\usepackage{cleveref}

\newtheorem{proposition}{Proposition}[section]
\newtheorem{lemma}[proposition]{Lemma}

\newtheorem{theorem}[proposition]{Theorem}
\newtheorem{definition}[proposition]{Definition}
\newtheorem{remark}[proposition]{Remark}

 \newcommand{\reg}{\mathrm{reg}}
\newcommand{\depth}{\mathrm{depth}}
\renewcommand{\dim}{\mathrm{dim}}
\renewcommand{\l}{\langle}
\renewcommand{\r}{\rangle}
\newcommand{\im}{\mathrm{im}}
\newcommand{\ini}{\mathrm{in}}
\newcommand{\supp}{\mathrm{supp}}
\newcommand{\cone}{\mathrm{cone}}
\newcommand{\iv}{\mathrm{iv}}
\newcommand{\G}{\mathcal{G}}
\newcommand{\C}{\mathcal{C}}
\usepackage{amscd}
\usepackage{graphicx}
\usepackage[all]{xy}
\usepackage{graphicx}
\usepackage{tikz}

\title[Generalized binomial edge ideals of whisker graphs]{Generalized binomial edge ideals of whisker graphs via an extension of generalized corona products}
\author{
J. Anuvinda
\and
Ranjana Mehta
\and
Kamalesh Saha
}
\date{}

\address{Department of Mathematics, SRM University-AP, Amaravati 522240, Andhra Pradesh, India}
\email{\url{anuvinda_j@srmap.edu.in}}

\address{Department of Mathematics, SRM University-AP, Amaravati 522240, Andhra Pradesh, India} 
\email{\url{ranjana.m@srmap.edu.in}}

\address{Department of Mathematics, SRM University-AP, Amaravati 522240, Andhra Pradesh, India}
\email{\url{kamalesh.saha44@gmail.com}; \url{kamalesh.s@srmap.edu.in}}



\allowdisplaybreaks

\begin{document}

\begin{abstract}
In this paper, we initiate a systematic study of generalized binomial edge ideals of whisker graphs by working within a substantially broader class of graphs. We extend the notion of generalized corona products, and through this enlarged framework, investigate fundamental algebraic invariants such as depth, (Castelnuovo–Mumford) regularity, and the Cohen–Macaulay property. In particular, we establish a sharp lower bound on the depth of generalized binomial edge ideals for our extended class, and further obtain explicit depth formula for a broad subclass of this family, which in turn recovers the depth formula for whisker graphs. We also establish sharp upper bounds for the regularity, and in the case of binomial edge ideals of whisker graphs over gap-free graphs, determine the exact value of the regularity. Finally, for our extended class, we provide a combinatorial classification of all Cohen–Macaulay binomial edge ideals, which in turn yields a new construction of Cohen–Macaulay binomial edge ideals.
\end{abstract}

\maketitle

\section{Introduction}\label{sec:intro}

Let $G=(V(G),E(G))$ be a simple graph on the vertex set $V(G)=[n]=\{1,\dots,n\}$. Consider the polynomial ring $R = K[x_1,\dots,x_n,y_1,\dots,y_n]$ over a field $K$. The \emph{binomial edge ideal} of $G$ is the ideal
\[
J_G = \left\langle x_i y_j - x_j y_i \mid \{i,j\}\in E(G) \text{ with } i<j\right\rangle
\subset R.
\]
This class of ideals was introduced independently by Herzog et al. \cite{hhhkr10}, and Ohtani \cite{ohtani11}. Rauh \cite{rauh13} later introduced a natural extension of these ideals, known as the \emph{generalized binomial edge ideals}, which was further generalized by Ene-Herzog-Hibi-Qureshi \cite{ehhq14} as binomial edge ideals of pairs of graphs. Let $m,n$ be positive integers and set $R = K[x_{i,j} \mid i\in[m],\ j\in[n]]$. For simple graphs $G_1$ on $[m]$ and $G_2$ on $[n]$, and edges $e=\{i,j\}\in E(G_1)$, $f=\{k,l\}\in E(G_2)$ with $i<j$ and $k<l$, assign the $2$-minor $p_{(e,f)}=[i\ j\mid k\ l]=x_{ik}x_{jl}-x_{il}x_{jk}$. The \emph{generalized binomial edge ideal} of the pair $(G_1,G_2)$ is
\[
J_{G_1,G_2} = \left\langle p_{(e,f)} \mid e\in E(G_1), f\in E(G_2)\right\rangle\subset R.
\]
When $G_1=K_m$ is a complete graph, $J_{K_m,G_2}$ becomes a radical ideal, called the generalized binomial edge ideal associated to $G_2$, introduced in \cite{rauh13}, and the case $m=2$ recovers the usual binomial edge ideal of $G_2$.\par

Over the past fifteen years, there has been a significant surge of interest in the study of binomial edge ideals as well as the generalized one. In comparison to binomial edge ideals, generalized binomial edge ideals remain relatively less explored. For comprehensive accounts on binomial edge ideals, we refer the reader to the survey articles, books, and theses \cite{das_survey_binom, hho_book_binom, laclair_thesis, madani_survey}. For developments concerning generalized binomial edge ideals see \cite{acr26, ams24, ams25, ci20, kats25, kumar20, rauh13, sz24, sz25}. Two homological invariants that play a central role in our work are depth and Castelnuovo-Mumford regularity (in short, regularity). For a finitely generated graded $R$-module $M$, the \textit{depth} is defined as
\[
\depth(M)= \min\{ i \mid H^i_{\mathfrak m}(M)\neq 0\},
\]
where $H^i_{\mathfrak m}(M)$ denotes the $i^{\text{th}}$ local cohomology module of $M$ with respect to the unique homogeneous maximal ideal $\mathfrak m$ of $R$. The \textit{regularity} of $M$ is defined as
\[
\reg(M)=\max\{ j-i \mid \beta_{i,j}(M)\neq 0\},
\]
where $\beta_{i,j}(M)$ are the graded Betti numbers of $M$. By the Auslander-Buchsbaum formula, knowing the depth is equivalent to knowing the projective dimension. Since projective dimension together with regularity determine the extent of the Betti table, depth and regularity play a crucial role in estimating the size of minimal free resolutions.\par 

This paper deals with the depth, regularity, and Cohen-Macaulay property of (generalized) binomial edge ideals of whisker graphs, considering a much broader class of graphs, arising from an extension of generalized corona products. In \cite{hs25} and \cite{ks-m15}, binomial edge ideals of corona product have been considered with some specific base graph, whereas in this paper we allow the base graph to be arbitrary. Whisker graphs occupy a classical and important place in the literature of monomial edge ideals. The \textit{whisker} graph of $G$, denoted by $W(G)$, is obtained by attaching a new pendant (whisker) vertex to every vertex of $G$. In his foundational paper on Cohen-Macaulay graphs \cite{Vil90}, Villarreal showed that such whisker constructions produce the first natural family of graphs whose (monomial) edge ideals are Cohen-Macaulay, thereby establishing whisker graphs as a fundamental class in the theory of (monomial) edge ideals. Beyond commutative algebra, whisker graphs also play an essential role in graph theory and related areas, which motivates the broader generalizations considered in this paper. \par


Despite their significance in the study of monomial edge ideals, a systematic study of binomial edge ideals of whisker graphs appears to be largely absent in the literature. One of the reasons is structural: a key operation in the study of binomial edge ideals is the vertex operation $G_v$, which replaces a vertex $v$ by making all its neighbors adjacent to each other so that $v$ becomes a free vertex. This operation is particularly useful because it appears in a short exact sequence that relates $J_G$ to binomial edge ideals of graphs with fewer non-free vertices. Unfortunately, the $G_v$-operation does not preserve the class of whisker graphs in general, so the standard reduction techniques based on $G_v$ (and related decompositions) are not compatible with whiskering. This incompatibility explains why methods that work for many other graph classes may fail for whisker graphs in the binomial context. 

In this paper, we address this gap and provide the first systematic study of binomial and generalized binomial edge ideals for whisker graphs by considering a much broader class, which even extends the generalized corona products of graphs. Let $G$ be a graph on $[n]$ and $H_1,\ldots,H_n$ be a sequence of $n$ graphs. Then the generalized corona product of $G$ and $H_1,\ldots,H_n$, denoted by $G\circ (H_1,\ldots,H_n)$, is a graph obtained by attaching the $i^{\text{th}}$ vertex of $G$ to every vertex of $H_i$. If $H_1=H_2=\cdots=H_n=H$, then it gives the usual corona product of two graphs $G$ and $H$, denoted by $G\circ H$. Again, observe that if $H=K_1$ (a singleton vertex), then $G\circ K_1$ is nothing but the whisker graph $W(G)$ of $G$. In the context of generalized binomial edge ideals, we extend all these classes, thereby obtaining a key technique to study them in a broader framework. \par

Let us decompose the vertex set of a graph $G$ into free and non-free vertices such that $V(G)=A_G\sqcup B_G$, where $A_G$ and $B_G$ denote the free and non-free vertices of $G$, respectively. We introduce two new classes of graphs:
$$\G_1= \{W_S(G)\mid B_G\subset S\subset V(G)\}$$
and
$$\G_2=\{G\circ_{S} (H_1,\ldots,H_{\vert S\vert})\mid B_G\subset S\subset V(G)\},$$
  where $W_S(G)$ is obtained by attaching a whisker to each vertex in a chosen set $S$, and $G\circ_{S}(H_1,\ldots,H_{\vert S\vert})$ is obtained by attaching arbitrary graphs $H_j$ to chosen vertices $v_j\in S$. Then one can observe the following containment
  $$\text{Whisker Graphs }\subset \text{ Corona Product }\subset \text{ Generalized Corona Product }\subset\, \G_2$$
  and 
  $$\text{Whisker Graphs }\subset\, \G_1\,\subset \,\G_2.$$

  In this paper, we first investigate the depth of generalized binomial edge ideals of graphs belong to the class $\G_2$. We give a tight lower bound on depth for this class as follows:
  \begin{customthm}{\ref{thm:depth-bound-gen}}
Let $D =G\circ_{S}(H_1,\ldots,H_{\ell})$ be a graph in $\mathcal{G}_{2}$. Then $$ \depth(R/J_{K_{m},D})\geq \sum_{i=1}^{\ell}(f(H_{i})+d(H_{i}))+ p-\ell+(m-1)c(D),$$ 
where $ \ell $ is the cardinality of $S$, $p$ is the number of vertices in $G$, $c(D)$ is the number of connected components of $D$, $f(H_{i})$ denotes the number of free vertices in $H_{i} $ and $ d(H_{i})$ denotes the sum of diameters of connected components of $H_{i}$.
\end{customthm}
\noindent Next, for a large subclass of $\G_2$ (which includes whisker graphs, the generalized corona product of $G$ with the attached graphs having Cohen-Macaulay generalized binomial edge ideals, etc.), we establish an exact formula for the depth. Consider the subclass $ \mathcal{G}' \subset \mathcal{G}_{2}$ by imposing an additional condition on the attached graphs $H_{i}$ such that $\depth(R_i/J_{K_m,H_i})=m+\vert V(H_i)\vert-1$ for each $i$. We establish the following for the class $\G'$.

\begin{customthm}{\ref{thm:depth-equal}}
    Let $D$ be any graph belong to the class $\mathcal{G}'$. Then 
    $$ \depth(R/J_{K_{m},D})= \vert V(D) \vert + (m-1)c(D).$$
    In particular, if $D=G\circ_{S}(H_1,\ldots,H_{\vert S\vert})\in \G_2$ and each $H_i$ is connected with Cohen-Macaulay $J_{H_i}$, then
    $$\depth(R/J_D)=\vert V(D)\vert+c(D).$$
    Hence, $\depth(R/J_{W(G)})=2\vert V(G)\vert+c(G)$ for any graph $G$.
\end{customthm}
Finally, we improve the lower bound of the depth for the binomial edge ideals of graphs in $\G_2$ compared to the bound obtained for the generalized case in \Cref{thm:depth-bound-gen}, as follows:

\begin{customthm}{\ref{thm:depth-bound-binom}}
    Let $D =G\circ_{S}(H_1,\ldots,H_{\ell}) \in \mathcal{G}_{2} $ be a simple graph. Then $$ \depth(R/J_{D}) \geq \sum_{i=1}^{\ell}\depth(R_{i}/J_{H_{i}})+ p-\ell+c(D),$$
where $ \ell$ is the cardinality of $ S $, $p$ is the number of vertices in $G$.
\end{customthm}

In \Cref{sec:reg}, we focus on studying the regularity of binomial edge ideals of graphs belong to $\G_1$. We believe that the same study for the class $\G_2$ is challenging unless we restrict the base graph to be very special. First, we establish a tight upper bound on the regularity using the induced matching number of the base graph, which significantly improves the existing general upper bounds.

\begin{customthm}{\ref{thm:reg-bound-gen}}
   Let $H=W_{S}(G)$ be any graph belong to the class $\G_1$. Then 
$$\reg(R/J_{K_{m},H}) \leq (m-1)( \vert S \vert + \im(G)).$$
In particular, for any graph $G$, we have
$$\reg(R/J_{W(G)})\leq \vert V(G)\vert+\im(G).$$
\end{customthm}
\noindent Here, $\im(G)$ denotes the induced matching number of $G$. This invariant plays a central role in the study of regularity of monomial edge ideals, providing a general lower bound and achieving equality for many graph classes. Our result shows that it is likewise significant in estimating the regularity of generalized binomial edge ideals of whisker graphs. Next, to demonstrate the sharpness of the upper bound, we explicitly compute the regularity of binomial edge ideals of whisker graphs on gap-free graphs by analyzing the initial ideal via the reduced Gr\"obner basis technique.
\begin{customthm}{\ref{thm:reg-bound-gapfree}}
    Let $G$ be a gap-free graph. Then 
    $$\reg(R/J_{W(G)})=\vert V(G)\vert+1.$$
\end{customthm}

The final section of this article is devoted to the Cohen–Macaulay classification of graphs in the class $\G_2$. Although there are numerous works on classifying Cohen–Macaulay binomial edge ideals (see \cite{bms18, bmrs24, ehh11, ks-m15, lmr23, rr14, ss25_closed, ss25_binom-whisker}), a complete classification remains largely open and continues to attract interest, particularly in view of the conjecture proposed in \cite{bms22}. Notably, our classification provides a new construction of Cohen–Macaulay binomial edge ideals.

\begin{customthm}{\ref{thrm:cohen-macly}}
    Let $D=G\circ_{S}(H_1,\ldots,H_{\vert S\vert})$ be any connected graph in $\G_2$ with $G$ non-empty. Then the following are equivalent:
    \begin{enumerate}
        \item $J_{D}$ is Cohen-Macaulay;
        \item $G$ is complete, each $H_i$ is connected with Cohen-Macaulay $J_{H_i}$, and whenever $\vert S\vert=\vert V(G)\vert$ at least one of $H_{i}$'s is complete.
        \end{enumerate}
\end{customthm}
\noindent Since for $m>2$, $J_{K_m,G}$ is Cohen-Macaulay if and only if $G$ is complete \cite[Corollary 4.3]{acr26}, the cases $m\geq 3$ offer no further interest for the classification of the Cohen–Macaulay property.

\section{Preliminaries}

In this section, we discuss some concepts and results that will be required for the subsequent sections. The notion of generalized binomial edge ideals $J_{K_m,G}$ is defined for $m\geq 2$. So, without mentioning we will assume $m\geq 2$ and $G$ to be non-empty.

Consider a simple graph $G$ on $ [n] $. Let $T \subset [n]$. We write $G-T$ to denote the induced subgraph on the vertex set $V(G)\setminus T$. In particular by $G-v$, we denote the induced subgraph $G-\{v\}$ for any $v \in V(G)$. Let $c_{G}(T)$ denote the number of connected components of $G-T$, i.e. $c_{G}(T)=c(G-T)$. A vertex $v$ is called a \textit{cut vertex} of $G$ if the number of connected components of $G$ is less than that of $G-v$. If $v$ is a cut vertex of the induced subgraph $G-(T\setminus \{v\})$ for every $v \in T$, then $T$ is said to be a \textit{cutset} of $G$. We denote by $\C(G)$ the set of all cutsets of $G$, and by definition, we have $\emptyset\in\C(G)$.\par  

For a graph $G$, we write $\tilde{G}$ to represent the complete graph on $V(G)$. Let $G_{1},\ldots,G_{c_{G}(T)}$ be the connected components of $G-T$ for some $T\subset [n]$. Now, for an integer $m\geq 2$ and $T\subset [n]$, let us consider the ideal 
$$P_{T}(K_{m},G)=\left\l \{x_{ij}\mid (i,j)\in [m] \times T\},J_{K_{m},\tilde{G_{1}}},J_{K_{m},\tilde{G_{2}}},\ldots,J_{K_{m},\tilde{G}_{c_{G}(T)}}\right\r$$
in the polynomial ring $R=K[x_{ij}\mid i\in [m],j \in [n]]$.
Then $P_{T}(K_{m},G)$ is a prime ideal containing $J_{K_{m},G}$. From \cite[Corollary 4]{rauh13} it follows that $J_{K_m,G}$ is a radical ideal, and by \cite[Theorem 7]{rauh13}, the minimal prime ideals of $J_{K_{m},G}$ are precisely the prime ideals of the form $P_{T}(K_{m},G)$ such that $T\in \mathcal{C}(G)$. Thus, we have 
$$ J_{K_{m},G} = \bigcap_{T \in \mathcal{C}(G)}P_{T}(K_{m},G).$$

Let $G$ be a graph. The \textit{neighbor} set of a vertex $v\in V(G)$ is the set of vertices of $G$ that are adjacent to $v$ and is denoted by $N_{G}(v)$. We write $N_{G}[v]:=N_{G}(v)\cup\{v\}$. As we mentioned in the introduction, the graph $G_{v}$ has a particular importance in the theory of generalized binomial edge ideals, which explicitly has the following vertex and edge sets:
$$ V(G_{v})=V(G) \text{ and } E(G_{v})=E(G)\cup \{\{u,w\} \mid u, w \in N_{G}(v)\}.$$
A vertex $v \in V (G)$ is said to be a \textit{free} (or \textit{simplicial}) vertex of $G$ if the induced subgraph of $G$ on the vertex set $N_{G}(v)$ is a complete graph; equivalently, $G=G_v$. Otherwise, $v$ is called a \textit{non-free} (or \textit{internal}) vertex. We denote by $f(G)$ and $\iv(G)$ the number of free and non-free vertices of $G$, respectively; in particular, $f(G)=|A_{G}|$ and $\iv(G)=|B_G|$. An induced subgraph of $G$ is called a \textit{clique} if it is complete. Note that if $v$ is a free vertex of $G$, then the induced subgraph of $G$ on the vertex set $N_G(v)$ is a clique.\par

We now recall some fundamental results and facts that play a crucial role in the study of generalized binomial edge ideals.

\begin{theorem}[{\cite[Theorem 3.2] {kumar20}}] \label{thm:decomp}Let $G$ be a finite simple graph and $v$ be a non-free vertex of $G$. Then,$$J_{K_{m},G}=J_{K_{m},G_{v}} \cap \left\l \{x_{iv}\mid i\in [m]\},J_{K_{m},G-v}\right\r.$$ 
\end{theorem}
\begin{lemma} [{\cite[Lemma 3.4]{kumar20}}]\label{lem:non-free} Let $G$ be a graph and $v$ be a non-free vertex of $G$. Then $\mathrm{max}\{\iv(G_{v}),\iv(G-v),\iv(G_{v}-v)\} < \iv(G)$.\end{lemma}

\noindent It is easy to observe that 
$$J_{K_{m},G_{v}}+\l\{x_{iv}\mid i\in [m]\}, J_{K_{m},G-v}\r= \l\{x_{iv}\mid i\in [m]\}, J_{K_{m},G_{v}-v}\r.$$
Therefore, using the decomposition given in \Cref{thm:decomp}, we obtain the following short exact sequence
\begin{align*}
    0 \longrightarrow R/J_{K_{m},G} \longrightarrow R/J_{K_{m},G_{v}} &\oplus R/\l\{x_{iv}\mid i\in [m]\}, J_{K_{m},G-v}\r\\ 
&\longrightarrow R/\l\{x_{iv}\mid i\in [m]\}, J_{K_{m},G_{v}-v}\r \longrightarrow 0,
\end{align*}
which is equivalent to
\begin{equation}\label{shrt-exact-1}
  0 \longrightarrow R/J_{K_{m},G} \longrightarrow R/J_{K_{m},G_{v}} \oplus R_{v}/J_{K_{m},G-v} \longrightarrow R_{v}/J_{K_{m},G_{v}-v} \longrightarrow 0,
\end{equation}
where $R_{v}= K[x_{ij}\mid i\in [m],j\in V(G-v)]$.

\begin{remark}\label{rem:depth}{\rm
    Note that $J_{K_m,K_n}$ is the determinantal ideal generated by all $2\times 2$ minors of a generic matrix of order $m\times n$. Thus, by \cite[Theorem 1.10]{bc2003}, we get $\dim(R/J_{K_m,K_n})=n+m-1$. Now, for any connected graph $G$ on $[n]$, we have $\emptyset\in \C(G)$ and $P_{\emptyset}(K_m,G)=J_{K_m,K_n}$. Thus, if $G$ is connected and $J_{K_m,G}$ is Cohen-Macaulay, then $\depth(R/J_{K_m,G})=|V(G)|+m-1$. Since the depth of generalized binomial edge ideals is additive over the disjoint union of graphs, we have $\depth(R/J_{K_m,G})=|V(G)|+(m-1)c(G)$ whenever $J_{K_m,G}$ is Cohen-Macaulay.
    }
\end{remark}

\begin{theorem}[{\cite[Theorem 3.4]{ams24}}]\label{thm: depth gen lower bound}
    Let $G$ be a simple graph. Then we have
    $$\depth(R/J_{K_m,G})\geq f(G)+d(G)+(m-2)c(G).$$
\end{theorem}
\begin{theorem}[{\cite[Theorem 4.5]{ams24}}]\label{thm: depth gen upper bound}
Let $ G $ be a connected non-complete graph. If the vertex connectivity of $G$ is $\kappa(G)$, then
$$\depth(R/J_{K_m,G})\leq m+\vert V(G)\vert-\kappa(G).$$
In particular, it follows from \Cref{rem:depth} that for any connected graph $G$,
$$\depth(R/J_{K_m,G})\leq m+\vert V(G)\vert-1.$$
\end{theorem}

The following depth and regularity lemmas are well-known in the literature, and are frequently used in this paper.

\begin{lemma}\label{lem:2} Let $M,N,P$ be finitely generated graded modules over $R$, and let \ $0\longrightarrow M \longrightarrow N \longrightarrow P \longrightarrow 0$ \ be a short exact sequence. Then \begin{enumerate}
\item $\depth (N) \geq \mathrm{min} \{\depth  (M), \depth  (P) \}$.
\\The equality holds if $\depth (P) \neq \depth (M)-1$.
\item $\depth  (M) \geq \mathrm{min}\{\depth  (N), \depth  (P) + 1 \}$.
\\The equality holds if $\depth (N) \neq \depth (P)$.
\item $\depth  (P) \geq \mathrm{min}\{\depth  (M)-1, \depth  (N)\}$.
\\The equality holds if $\depth (M) \neq \depth (N)$.
\end{enumerate}
\end{lemma}
\begin{lemma}\label{lem:4} Let $M,N,P$ be finitely generated graded modules over $R$, and let \ $0\longrightarrow M \longrightarrow N \longrightarrow P \longrightarrow 0$ \ be a short exact sequence. Then \begin{enumerate}
\item $\reg (M) \leq \mathrm{max} \{\reg (N), \reg (P)+1 \}$
\item $\reg (N) \leq \mathrm{max} \{\reg (M), \reg (P)\}$
\item $\reg (P) \leq \mathrm{max} \{\reg (M)-1, \reg (N)\}$
\end{enumerate}
\end{lemma}
\begin{lemma} \label{lem:6}
    Let $0\longrightarrow M \longrightarrow N \longrightarrow P \longrightarrow 0$ \ be a short exact sequence of modules over $R$. Then $\dim(N)=\mathrm{max}\{ \dim(M), \dim(P) \}$.
\end{lemma}

\section{Depth of generalized binomial edge ideals for $\G_2$}

In this section, we investigate the depth of generalized binomial edge ideals of graphs belong to the class $\G_2$. We give a sharp lower bound on the depth for this class, and establish the equality for a broad subclass of $\G_2$, which includes the class of whisker graphs. Recall the structure of graphs belong to the classes $\G_1$ and $\G_2$ (see \Cref{fig:g1} and \ref{fig:g2}).

\begin{figure}[H]
	\centering
\begin{tikzpicture}[scale=1]

\filldraw[black] (0,0) circle (2pt)node[anchor=north]{$1$};
\filldraw[black] (1,1) circle (2pt)node[anchor=south]{$2$};
\filldraw[black] (1,-1) circle (2pt)node[anchor=north]{$3$};
\filldraw[black] (2,0) circle (2pt)node[anchor=north]{$4$};
\filldraw[black] (-1,-1) circle (2pt)node[anchor=north]{$8$};
\filldraw[black] (3,-1) circle (2pt)node[anchor=north]{$9$};
\filldraw[black] (3,1) circle (2pt)node[anchor=south]{$10$};
\filldraw[black] (4,0) circle (2pt)node[anchor=north]{$5$};
\filldraw[black] (5,1) circle (2pt)node[anchor=south]{$6$};
\filldraw[black] (5,-1) circle (2pt)node[anchor=north]{$7$};

\draw[black] (0,0) -- (1,1) -- (2,0) -- (1,-1) -- cycle;

\draw[black] (4,0) -- (5,1) -- (5,-1) -- cycle;
\draw[black] (0,0) -- (2,0);
\draw[black] (2,0) -- (4,0);
\draw[red] (0,0) -- (-1,-1);
\draw[red] (2,0) -- (3,-1);
\draw[red] (4,0) -- (3,1);
\end{tikzpicture}
\caption{A graph in $\mathcal{G}_{1}$}\label{fig:g1}
\end{figure}
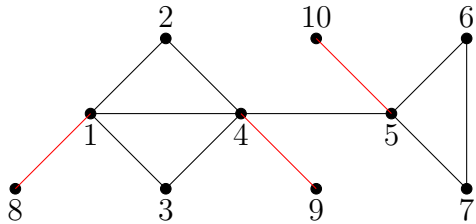

\begin{figure}[H]
	\centering
\begin{tikzpicture}[scale=1]

\filldraw[black] (0,0) circle (2pt)node[anchor=south]{$2$};
\filldraw[black] (1,1) circle (2pt)node[anchor=south]{$1$};
\filldraw[black] (0,-1) circle (2pt)node[anchor=north]{$3$};
\filldraw[black] (2,-1) circle (2pt)node[anchor=west]{$4$};
\filldraw[black] (2,0) circle (2pt)node[anchor=south]{$5$};
\filldraw[black] (-2,0.5) circle (2pt)node[anchor=south]{$6$};
\filldraw[black] (-1,1) circle (2pt)node[anchor=south]{$7$};
\filldraw[black] (3,1.5) circle (2pt)node[anchor=south]{$8$};
\filldraw[black] (4,1.5) circle (2pt)node[anchor=south]{$9$};
\filldraw[black] (4.5,0.5) circle (2pt)node[anchor=north]{$10$};
\filldraw[black] (3.5,0.75) circle (2pt)node[anchor=north]{$11$};
\filldraw[black] (2,-2) circle (2pt)node[anchor=north]{$12$};
\filldraw[black] (3,-2.5) circle (2pt)node[anchor=north]{$13$};
\filldraw[black] (4,-2) circle (2pt)node[anchor=north]{$14$};

\draw[blue] (3,1.5) -- (4,1.5) -- (4.5,0.5) -- (3.5,0.75) -- cycle;
\draw[black] (0,0) -- (1,1) -- (2,0) -- (2,-1) -- (0,-1) -- cycle;
\draw[black] (0,0) -- (2,-1);
\draw[black] (0,0) -- (2,0);
\draw[blue] (-2,0.5) -- (-1,1);
\draw[blue] (2,-2) -- (3,-2.5);
\draw[blue] (3,-2.5) -- (4,-2);
\draw[red] (0,0) -- (-2,0.5);
\draw[red] (0,0) -- (-1,1);
\draw[red] (2,0) -- (3,1.5);
\draw[red] (2,0) -- (4,1.5);
\draw[red] (2,0) -- (4.5,0.5);
\draw[red] (2,0) -- (3.5,0.75);
\draw[red] (2,-1) -- (2,-2);
\draw[red] (2,-1) -- (3,-2.5);
\draw[red] (2,-1) -- (4,-2);

\end{tikzpicture}
\caption{A graph in $\mathcal{G}_{2}$}\label{fig:g2}
\end{figure}
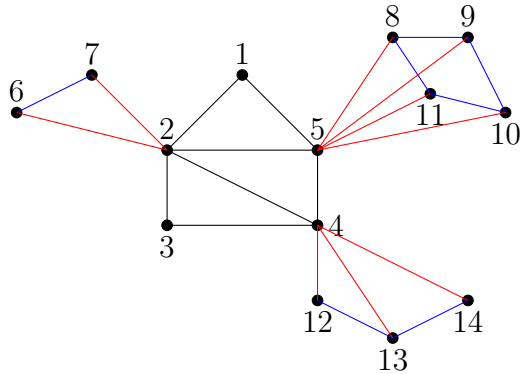

Let us start with the following lemma, which is crucial for the subsequent sections.

\begin{lemma}\label{lem:compatible class}
    Let $D=G\circ_{S}(H_1,\ldots,H_{\vert S\vert})$ be any graph in the class $\G_2$. Then $D-v_i=H_i\sqcup D'$, where $D'\in \G_2$ and $D_{v_i},D_{v_i}-v_i\in \G_2$ for any $v_i\in S$. In particular, if $D\in \G_1$, then $D-v_i=K_1\sqcup D'$, where $D'\in\G_1$ and $D_{v_{i}},D_{v_{i}}-v_i\in \G_1$ for any $v_i\in S$.
\end{lemma}
\begin{proof}
Let $S=\{v_1,\ldots,v_{\ell}\}$, and without loss of generality, assume that $v_i=v_{\ell}$. \\
\noindent \textbf{Claim-1:} $D-v_{\ell}=H_{\ell}\sqcup D'$, where $D'\in \G_2$.\\
\textit{Proof of Claim-1.} Let us consider the graph $D'=G'\circ_{S\setminus\{v_{\ell}\}}(H_1,\ldots,H_{\ell-1})$, where $G'=G-v_\ell$. Then we have $D-v_{\ell}=H_{\ell}\sqcup D'$. Now it remains to show that $D'\in \G_2$. Since removing a vertex does not produce any new non-free vertex, we have $B_{G'}\subset B_{G}$. Since $B_G\subset S$ and $v_{\ell}\not\in V(G')$, we must have $B_{G'}\subset B_{G}\setminus\{v_{\ell}\}\subset S\setminus \{v_{\ell}\}$. Thus, from our construction of the class $\G_2$, one can see that $D'\in \G_2$.\\
\noindent\textbf{Claim-2:} $D_{v_{\ell}}, D_{v_{\ell}}-v_{\ell}\in \G_2$.\\
\textit{Proof of Claim-2.} Let us consider the induced subgraph $G''=D_{v_{\ell}}[V(G)\cup V(H_{\ell})]$ of $D_{v_{\ell}}$. Then we have $D_{v_{\ell}}=G''\circ_{S\setminus \{v_{\ell}\}}(H_1,\ldots,H_{\ell-1})$. Note that the vertex $v_{\ell}$ and all the vertices of $H_{\ell}$ become free vertices in $G''$. Thus, $B_{G''}\subset B_{G}\setminus \{v_{\ell}\}\subset S\setminus \{v_{\ell}\}$, and hence, $D_{v_{\ell}}\in \G_2$. Similarly, we have $D_{v_{\ell}}-v_{\ell}\in \G_2$ by considering the base graph $G'''=G''-v_{\ell}$.\par 

The second part can be proved by imitating the argument used for the first part.
\end{proof}

Now, let us recall the definition of the diameter of a graph.

\begin{definition}{\rm
    Let $G$ be a connected graph. Let $u,v \in V$ be any two distinct vertices of $G$. Then the \textit{distance} between $u$ and $v$, denoted by $d_{G}(u,v)$, is the length of the shortest path connecting $u$ and $v$. The \textit{diameter} of $G$, denoted by $\mathrm{diam}(G)$, is defined as 
    $$\mathrm{diam}(G):=\mathrm{max}\{d_{G}(u,v)\mid u,v \in V(G)\}.$$ Let $G$ be a graph with connected components $G_1,\ldots,G_{c(G)}$. We define $d(G):= i(G)+\sum_{i=1}^{c(G)}\mathrm{diam}(G_i)$, where $i(G)$ denotes the number of isolated vertices of $G$.
    }
\end{definition} 

Using the invariants $d(G)$ and $f(G)$ of a graph $G$, Malayeri-Madani-Kiani \cite{mmk22} gave a sharp lower bound on the depth of binomial edge ideals, which has further been generalized in \cite{ams24}. In the following theorem, using these two invariants, we establish a better lower bound on the depth of generalized binomial edge ideals for the class $\G_2$.

\begin{theorem}\label{thm:depth-bound-gen}
Let $D =G\circ_{S}(H_1,\ldots,H_{\ell})$ be a graph in $\mathcal{G}_{2}$. Then $$ \depth(R/J_{K_{m},D})\geq \sum_{i=1}^{\ell}(f(H_{i})+d(H_{i}))+ p-\ell+(m-1)c(D),$$ 
where $ \ell $ is the cardinality of $S$ and $p$ is the number of vertices in $G$.
\end{theorem}
\begin{proof} Let $ S=\{v_{1},\ldots,v_{\ell}\}\subset V(G)$ be such that each $v_{i}$ is made adjacent to all vertices in $H_{i}$ to form the graph $D$. It is enough to prove the inequality when $D$ is connected. We proceed by mathematical induction on $\vert S\vert=\ell$. If $S=\emptyset$, i.e., $\ell=0$, then $G$ contains no non-free vertices and no $H_{i}$'s are attached to $G$. In this case, we have $D=G$ is a complete graph on $p$ vertices, and thus, $ \depth(R/J_{K_{m},D})= m+p-1$ by \Cref{rem:depth}. Hence, the result holds for the base case. Now, let us take $\ell \geq 1$ and consider the graphs $D-v_{\ell},D_{v_{\ell}},D_{v_{\ell}}-v_{\ell}$. By \Cref{lem:compatible class}, we have $D-v_{\ell}=H_{\ell}\sqcup D'$ with $D'\in\G_{2}$ and $D_{v_{\ell}},D_{v_{\ell}}-v_{\ell}\in\G_2$. Also, it follows from the proof of \Cref{lem:compatible class} that exactly $(\ell-1)$ graphs $H_1,\ldots,H_{\ell-1}$ are attached with $(\ell-1)$ vertices of the base graphs in $D',D_{v_{\ell}}$ and $D_{v_{\ell}}-v_{\ell}$. Therefore, by the induction hypothesis, the inequality holds true for these graphs. Note that the graphs $D_{v_{\ell}},D_{v_{\ell}}-v_{\ell}$ are connected as we assume $D$ is connected, and $c(D')\geq c(D)=1$. For simplicity, we write $v=v_{\ell}$. Hence, we have
\begin{align*}
    \depth(R_v/J_{K_{m}, D-v}) &\geq  \depth(R_{\ell}/J_{K_{m}, H_{\ell}}) +\sum_{i=1}^{\ell-1}(f(H_{i})+d(H_{i}))\\ & \hspace{4cm}+ (p-1)- (\ell-1)+(m-1)c(D')\\
& \geq f(H_{\ell})+d(H_{\ell})+(m-2)c(H_{\ell})+ \sum_{i=1}^{\ell-1}(f(H_{i})+d(H_{i}))\\
& \hspace{3cm}+ p-\ell+(m-1) \quad (\text{by \Cref{thm: depth gen lower bound}})\\
& \geq \sum_{\substack{i=1}}^{\ell}(f(H_{i})+d(H_{i}))+ p-\ell+(m-1),
\end{align*}
where $R_{v}$ and $R_{\ell}$ are the corresponding polynomial rings of $J_{K_m,D-v}$ and $J_{K_m,H_{\ell}}$ respectively. Again, we have
\begin{align*}
\depth(R/J_{K_{m}, D_{v}}) & \geq \sum_{i=1}^{\ell-1}(f(H_{i})+d(H_{i}))+ (p+\vert V(H_{\ell})\vert)-(\ell-1)+ (m-1).
\end{align*}
Since by \Cref{thm: depth gen lower bound} and \ref{thm: depth gen upper bound}, $f(H_{\ell})+d(H_{\ell})\leq \vert V(H_{\ell})\vert+1$, we have
\begin{align*}
   \depth(R/J_{K_{m}, D_{v}})\geq \sum_{i=1}^{\ell}(f(H_{i})+d(H_{i}))+ p-\ell+ (m-1). 
\end{align*}
Similarly, we get
\begin{align*}
\depth(R_{v}/J_{K_{m},D_{v}-v}) & \geq \sum_{i=1}^{\ell}(f(H_{i})+d(H_{i}))+ p-\ell+ (m-1)-1. 
\end{align*}
Now considering the short exact sequence (\ref{shrt-exact-1}) and using \Cref{lem:2}, we obtain
\begin{align*} \depth (R/J_{K_m,D}) &\geq \mathrm{min}\{\depth(R_{v}/J_{K_{m},D-v}), \depth(R/J_{K_{m},D_{v}}),\\& \hspace{2cm}\depth(R_{v}/J_{K_{m},D_{v}-v})+1 \} \\
 & \geq \sum_{\substack{i=1}}^{\ell}(f(H_{i})+d(H_{i}))+ p-\ell + (m-1).
\end{align*}
This completes the proof.
\end{proof}

We now define a subclass $ \mathcal{G}' \subset \mathcal{G}_{2} $ by imposing an additional condition on the attached graphs $H_{i}$. By \Cref{thm: depth gen upper bound}, any connected graph $H$ satisfies $\depth(R/J_{K_m,H})\leq m+\vert V(H)\vert -1$ and this bound is sharp. For a fixed $m\geq 2$, the class $\G'$ consists of those graphs $D=G\circ_{S}(H_1,\ldots,H_{\vert S\vert})\in \G_{2}$ such that $ \depth(R_{i}/J_{K_{m},H_{i}})=m+\vert V(H_i)\vert - 1$ for each $i$, where $R_i=K[x_{i,j}\mid i\in [m], j\in V(H_i)]$. The following theorem determines the depth formula for the generalized binomial edge ideals of graphs belong to $\G'$, and so for the generalized binomial edge ideals of whisker graphs.

\begin{theorem}\label{thm:depth-equal}
Let $D$ be any graph belong to the class $\mathcal{G}'$. Then 
    $$ \depth(R/J_{K_{m},D})= \vert V(D) \vert + (m-1)c(D).$$
    In particular, if $D=G\circ_{S}(H_1,\ldots,H_{\vert S\vert})\in \G_2$ and each $H_i$ is connected with Cohen-Macaulay $J_{H_i}$, then
    $$\depth(R/J_D)=\vert V(D)\vert+c(D).$$
    Hence, $\depth(R/J_{W(G)})=2\vert V(G)\vert+c(G)$ for any graph $G$.
\end{theorem}
\begin{proof}
The proof uses ideas analogous to those in the previous proof, with necessary modifications. Consider $S=\{v_{1},\ldots,v_{\ell}\}\subset V(G)$, where each $v_{i}$ is adjacent to all vertices in $H_{i}$ forming the graph $D$. We prove the formula using mathematical induction on $\vert S \vert = \ell$. If $S=\emptyset $, then $D=G$ is a complete graph on $p$ vertices. Thus, $ \depth(R/J_{K_{m},D})=m+p-1$ by \Cref{rem:depth}. Hence, the formula holds for the base case. Now let us take $ \ell \geq 1 $ and consider the graphs $ D-v,D_{v},D_{v}-v$, where $v=v_{\ell}$. By \Cref{lem:compatible class}, we have $D-v=H_{\ell}\sqcup D'$  with $D'\in \mathcal{G}'$, and $D_{v},D_{v}-v\in \mathcal{G}'$. Moreover it follows from the proof of \Cref{lem:compatible class} that exactly $(\ell-1)$ graphs $H_1,\ldots,H_{\ell-1}$ are attached with $(\ell-1)$ distinct vertices of the base graphs in $D',D_v,D_{v}-v$. Therefore, by the induction hypothesis, the equality holds for these graphs. Consequently, we obtain 
\begin{align}
\depth(R_{v}/J_{K_{m},D-v}) & = \depth(R_{\ell}/J_{K_{m},H_{\ell}})+\depth(R'/J_{K_m,D'}) \nonumber \\
& =  (m+\vert V(H_{\ell})\vert-1)+|V(D')|+(m-1)c(D') \nonumber \\
& \geq |V(D)|+(m-1)c(D)+(m-2)\quad (\text{since $c(D')\geq c(D)$}) \nonumber \\ 
& \geq |V(D)|+(m-1)c(D),\label{eq3.1}
\end{align}
where $R_{v}$, $R_{\ell}$ and $R'$ are the corresponding polynomial rings of $J_{K_m,D-v}$, $J_{K_m,H_{\ell}}$ and $J_{K_m,D'}$, respectively. Now, from the construction of $D_v$, one can see that $|V(D_v)|=|V(D)|$ and $c(D_{v})=c(D)$. Thus, we have
\begin{align}
\depth(R/J_{K_{m}, D_{v}}) & = |V(D_v)|+(m-1)c(D_v) \nonumber \\
& = |V(D)|+(m-1)c(D) \label{eq3.2}
\end{align}
Similarly, for the graph $D_{v}-v$, we obtain
\begin{align}
\depth(R_{v}/J_{K_{m}, D_{v}-v}) & =  |V(D_{v}-v)|+(m-1)c(D_{v}-v) \nonumber\\
&= |V(D)|+(m-1)c(D)-1\label{eq3.3}
\end{align}
Now, from Equations (\ref{eq3.1}), (\ref{eq3.2}) and (\ref{eq3.3}), we have
\begin{align*}
\depth((R_{v}/J_{K_{m}, D-v}) \oplus (R/J_{ K_{m},D_{v}})) & = \mathrm{min}\{ \depth(R_{v}/J_{K_{m}, D-v}), \depth(R/J_{K_{m}, D_{v}})\} \\
& =  \depth(R/J_{K_{m}, D_{v}}) \\
& \neq  \depth(R_{v}/J_{ K_{m},D_{v}-v})
 \end{align*}
Hence, from the short exact sequence (\ref{shrt-exact-1}) and \Cref{lem:2}(2), we get
 \begin{align*} \depth (R/ J_{K_{m},D}) &  =  \mathrm{min}\{ \depth(R/J_{K_{m},D_{v}}), \depth(R_{v}/J_{K_{m},D_{v}-v})+1 \} \\
 & =|V(D)|+(m-1)c(D).
\end{align*}
The moreover part follows since $|V(W(G))|=2|V(G)|$ and $c(W(G))=c(G)$.
\end{proof}

Next, we improve the lower bound for the depth of binomial edge ideals of graphs in $\G_2$ compared to the bound proved for generalized binomial edge ideals in \Cref{thm:depth-bound-gen}.

\begin{theorem}\label{thm:depth-bound-binom}
Let $D =G\circ_{S}(H_1,\ldots,H_{\ell}) \in \mathcal{G}_{2} $ be a simple graph. Then $$ \depth(R/J_{D}) \geq \sum_{i=1}^{\ell}\depth(R_{i}/J_{H_{i}})+ p-\ell+c(D),$$
where $ \ell$ is the cardinality of $ S $, $p$ is the number of vertices in $G$.
\end{theorem}
\begin{proof}
Let $ S=\{ v_{1},\ldots,v_{\ell} \} \subset V(G)$ be such that each $ v_{i} $ is adjacent to all vertices in $ H_{i}$ to form $D$. Since the depth of a binomial edge ideal, the quantity $(p-\ell)$, and the number of connected components are additive over the disjoint union of graphs, without loss of generality, we may assume $D$ is connected. We proceed by mathematical induction on $\vert S \vert = \ell$. If $S=\emptyset $, then $B_{G}=\emptyset$ and no $H_{i}$'s are attached to $G$. In this case, $D = G$ is a complete graph with $p$ vertices. Therefore, $\depth(R/J_{D})=p+1$ by \Cref{rem:depth}. Hence, the result is true for the base case. Now, let us assume $\ell \geq 1$ and consider the graphs $D-v,D_{v},D_{v}-v$, where $v=v_{\ell}$. By \Cref{lem:compatible class}, we have $D-v=H_{\ell}\sqcup D'$ with $D'\in\G_{2}$ and $D_{v},D_{v}-v\in \G_2$. Again from the proof of \Cref{lem:compatible class}, it follows that exactly $(\ell-1)$ graphs $H_1,\ldots,H_{\ell-1}$ are attached with the base graphs in $D',D_{v}$ and $D_{v}-v$. Thus, by the induction hypothesis, these graphs follow the inequality. In addition, the graphs $D_{v},D_{v}-v$ are connected as $D$ is connected, and $c(D')\geq c(D)=1$. Hence, we have
\begin{align*}\depth(R_{v}/J_{ D-v}) & = \depth(R_{\ell}/J_{H_{\ell}})+\depth(R'/J_{D'})\\
& \geq \depth(R_{\ell}/J_{H_{\ell}})+\sum_{\substack{i=1}}^{\ell-1}\depth(R_{i}/J_{H_{i}})+ (p-1) \\
&  \hspace{4cm} -(\ell-1)+c(D')\\
& \geq \sum_{\substack{i=1}}^{\ell}\depth(R_{i}/J_{H_{i}})+ p-\ell+1,
\end{align*}
where $R_{v},R_{i}$ and $R'$ are the corresponding polynomial rings of $J_{D-v}$, $J_{H_{i}}$ and $J_{D'}$, respectively. Again, for the graph $D_v$, we get
\begin{align*}
\depth(R/J_{ D_{v}}) & \geq \sum_{\substack{i=1}}^{\ell-1}\depth(R_{i}/J_{H_{i}})+ (p+\vert V(H_{\ell})\vert)-(\ell-1)+1 \\
& =  \sum_{\substack{i=1}}^{\ell-1}\depth(R_{i}/J_{H_{i}})+ p+\vert V(H_{\ell})\vert-\ell+2 \\
& \geq  \sum_{\substack{i=1}}^{\ell}\depth(R_{i}/J_{H_{i}})+ p-\ell+1 \  (\text{From \Cref{thm: depth gen upper bound}})
\end{align*}
In the same way, we have
\begin{align*}
\depth(R_{v}/J_{ D_{v}-v}) & \geq  \sum_{\substack{i=1}}^{\ell-1}\depth(R_{i}/J_{H_{i}})+ (p+\vert V(H_{\ell})\vert-1)-(\ell-1)+1 \\
& =  \sum_{\substack{i=1}}^{\ell-1}\depth(R_{i}/J_{H_{i}})+ p+\vert V(H_{\ell})\vert-\ell+1 \\
& \geq  \sum_{\substack{i=1}}^{\ell}\depth(R_{i}/J_{H_{i}})+ p-\ell \quad  (\text{From \Cref{thm: depth gen upper bound}})
\end{align*}
Now applying \Cref{lem:2} to the short exact sequence (\ref{shrt-exact-1}) and using the above inequalities, we get 
\begin{align*} \depth(R/J_{D}) &  \geq  \mathrm{min}\{ \depth(R_{v}/J_{D-v}), \depth(R/J_{D_{v}}),\\ &  \hspace{2cm}  \depth(R_{v}/J_{D_{v}-v})+1 \} \\
 & \geq  \sum_{\substack{i=1}}^{\ell}\depth(R_{i}/J_{H_{i}})+ p-\ell+1
\end{align*}
Hence, in general, $\depth (R/ J_{D}) \geq  \sum_{i=1}^{\ell}\depth(R_{i}/J_{H_{i}})+ p-\ell+c(D)$.
\end{proof}

\section{Regularity of generalized binomial edge ideals for $\G_1$} \label{sec:reg}

This section investigates the regularity of generalized binomial edge ideals for the class $\G_1$. We establish a sharp upper bound on the regularity in terms of the induced matching number of the base graph, which improves the existing general upper bounds in the case of $\G_1$; consequently, in the case of whisker graphs. Furthermore, we explicitly find the regularity of binomial edge ideals of whisker graphs on gap-free graphs by analyzing the initial ideal via the reduced Gr\"obner basis technique. We begin with the definition of an induced matching.

\begin{definition}{\rm
    A matching in a graph $ G$ is a subgraph consisting of pairwise disjoint edges. A matching is called an \textit{induced matching} if no two of the edges in the matching are joined by an edge in $ G$. The largest size of an induced matching in $ G$ is called the \textit{induced matching number} of $G$, denoted by $\im(G)$.
    }
\end{definition}

Note that a graph $G$ is said to be \textit{gap-free} if $\im(G)=1$. Any co-chordal graph is an example of a gap-free graph.

\begin{theorem}\label{thm:reg-bound-gen}
Let $H=W_{S}(G)$ be any graph belong to the class $\G_1$. Then 
$$\reg(R/J_{K_{m},H}) \leq (m-1)( \vert S \vert + \im(G)).$$
In particular, for any graph $G$, we have
$$\reg(R/J_{W(G)})\leq \vert V(G)\vert+\im(G).$$
\end{theorem}
\begin{proof}
Since the regularity of generalized binomial edge ideals and induced matching numbers of graphs are additive over the disjoint union of graphs, we may assume $H$ to be connected. We prove by mathematical induction on the cardinality of $S$. If $ \vert S \vert=0$, then $B_{G}=\emptyset$, which implies $H=G$ is a complete graph. Hence, by \cite[Proposition 3.3]{kumar20}, $\reg (R/J_{K_{m},H})=\mathrm{min} \{ m-1, |V(H)|-1 \}$. Thus, the result holds for the base case. Now, suppose $\vert S\vert\geq 1$. Choose a vertex $v \in S$. Then by \Cref{lem:compatible class}, $H-v=K_1\sqcup H'$ with $H'\in \G_1$ and $H_{v}, H_{v}-v \in\G_1$. Note that $H',H_v,H_{v}-v$ have exactly $(\vert S\vert -1)$ whiskers attached with the base graphs $G',G'',G'''$, respectively. Hence, we can apply the induction hypothesis on these graphs. Consequently, we get 
\begin{align}
\reg(R_v/J_{K_m,H-v}) &=\reg(R_v/J_{K_m,H'}) \nonumber \\
& \leq (m-1)(\vert S \vert-1 + \im(G'));\label{eq1}\\
\reg(R/J_{K_m,H_{v}}) &\leq (m-1)(\vert S \vert-1 + \im(G''));\label{eq2} \\
\reg(R_v/J_{K_m,H_{v}-v}) & \leq (m-1)(\vert S \vert-1 + \im(G''')).\label{eq3}
\end{align}
Now, we want to compare $\im(G'),\im(G''),\im(G''')$ with $\im(G)$.
Since $G'=G-v$ is an induced subgraph of $G$, we have $\im(G') \leq \im(G)$. Now, observe that if $\{u,v\}$ is the whisker attach to $v$ in $H$, then $G''$ is the graph with 
$$V(G'')=V(G)\cup\{u\} \text{ and } E(G'')=E(G_v)\cup\{\{u,w\}:w\in N_G[v]\}.$$
Now, it is easy to see that $G'''=G''-v\cong G_v$. Suppose $M=\{ e_1,\ldots,e_r\} $ form an induced matching in $G''$ with $r=\im(G'')$. If $ e_i \in E(G)$ for all $i$, then $M$ is also an induced matching in $G$ as $E(G)\subset E(G'')$. Now, let $e_i \notin E(G)$ for some $ i$. Since the neighborhood of $v$ forms a clique in $G''$, that clique will contribute a maximum of one edge to the induced matching $M$. Thus, we should have $M\setminus \{ e_i \} \subset E(G)$. This case, together with the earlier case, implies $\im(G'') \leq \im(G)+1$. Next, suppose $ \im(G''') > \im(G) $. Let $T$ denote a maximum induced matching of $ G'''\cong G_v$. Then $T$ contains exactly one new edge that is not in $G$ and no other edges in $M$ will contain vertices from $N_{G}[v]$ as the new edges of $G_v$ belong to a clique in $G_v$ formed by the vertices $N_{G}[v]$. If the new edge is $\{a,b\}$, then we have $\{a,b\}\subset N_{G}(v)$. In $G_v$, we have $N_{G_v}(v)\subset N_{G_v}(w)$ for any $w\in N_{G_v}(v)$. Then it is easy to see that $(T \setminus \{\{a,b\}\}) \cup \{\{a,v\}\}$ forms an induced matching in $G$, which is a contradiction to the fact that $\im(G''') > \im(G)$. Consequently, $ \im(G''')=\im(G_v) \leq \im(G) $. Finally, we get the following relations:
\begin{align}
    \im(G')&\leq \im(G);\label{eq4}\\
    \im(G'')&\leq \im(G)+1;\label{eq5}\\
    \im(G''')&\leq \im(G).\label{eq6}
\end{align}
Now, using \Cref{eq1} to (\ref{eq6}), we have
\begin{align*}
    \reg(R_{v}/J_{K_{m}, H-v}) & \leq  (m-1)(\vert S \vert -1+ \im(G'))\\
& \leq (m-1)(\vert S \vert + \im(G))-(m-1)\\
 & < (m-1)(\vert S \vert + \im(G));\\
\reg(R /J_{K_{m}, H_{v}})
& \leq (m-1)(\vert S \vert -1+ \im(G''))\\
& \leq  (m-1)(\vert S \vert + \im(G));\\
\reg(R_{v}/J_{K_{m}, H_{v}-v}) & \leq (m-1)(\vert S \vert -1+ \im(G'''))\\
& \leq (m-1)(\vert S \vert + \im(G))-(m-1)\\
 & < (m-1)(\vert S \vert+ \im(G)).
\end{align*}
Now applying \Cref{lem:4} to the short exact sequence (\ref{shrt-exact-1}) and using the above inequalities, we obtain
\begin{align*}
\reg(R/J_{K_{m},H})& \leq  \mathrm{max }\{ \reg(R/J_{K_{m},H_{v}}), \reg(R_{v}/J_{K_{m},H-v}), \reg(R_{v}/J_{K_{m},H_{v}-v})+1\}\\
& \leq (m-1)(\vert S \vert + \im(G)). 
\end{align*}
This completes the proof.
\end{proof}

\begin{remark}{\rm
    Let $H=W_{S}(G)\in\G_1$ be a graph on $n$ vertices. Then observe that due to \cite[Theorem 3.7]{kumar20}, $\reg(R/J_{K_{m},H})=n-1$ if $m\geq n$. Again, by \Cref{thm:reg-bound-gen} and \cite[Theorem 3.6]{kumar20}, we have $\reg(R/J_{K_{m},H}) \leq \min\{(m-1)(\vert S \vert+ \im(G)),n-1\}$ if $m<n$.
    }
\end{remark}

Next, to establish the regularity formula of binomial edge ideals for the whisker graphs on gap-free graphs, we need some preliminaries on edge ideals of hypergraphs and the reduced Gr\"obner basis of binomial edge ideals. Let us illustrate those first.

\begin{definition}{\rm A (simple) \textit{hypergraph} $\mathcal{H}$ is a pair $(V(\mathcal{H}),E(\mathcal{H}))$, where $V(\mathcal{H})$ is a finite set, called the set of vertices of $\mathcal{H}$, and $E(\mathcal{H})$ is a collection of non-empty subsets of $V(\mathcal{H})$, called the edge set, such that no two distinct edges in $E(\mathcal{H})$ contain each other.
}    
\end{definition}

Consider a simple hypergraph $\mathcal{H}$ with vertex set $V(\mathcal{H})=\{x_1,\ldots,x_n\}$. For any $E \subset V(\mathcal{H})$, define $X_E:= \Pi_{x_{i} \in E} x_{i}$ a square-free monomial in $A=K[x_1,\ldots,x_n]$. Then the \textit{edge ideal} $I_{\mathcal{H}}$ of $\mathcal{H}$ is an ideal in $A$ defined as $I_{\mathcal{H}}=\langle X_E : E \subset V(\mathcal{H}) \rangle $.

\begin{definition}{\rm
 An \textit{induced matching} in a simple hypergraph $\mathcal{H}$, is a set of pairwise disjoint edges $e_1,\ldots,e_r$ such that the only edges of $\mathcal{H}$ contained in $\cup_{i=1}^{r}e_i$ are $e_1,\ldots,e_r$.
}
\end{definition}

\begin{theorem}[{\cite[Corollary 3.9]{mv12}}] \label{thrm:9} Let $\mathcal{H}$ be a simple hypergraph and $M=\{e_1,\ldots,e_r\}$ be an induced matching in $\mathcal{H}$. Then $\sum_{i=1}^{r}(\vert e_i \vert-1)\leq \reg(A/I_{\mathcal{H}})$.
\end{theorem}

\begin{definition}{\rm
   Let $G$ be a graph with $V(G)=[n]$. A path $ \pi : i = i_{0},i_{1},\ldots,i_{r} = j $ from $i$ to $j$ with $i < j$ in $G$ is said to be an \textit{admissible path} in $ G $ if the following hold:
\begin{itemize}
    \item $i_k \neq i_l $ for $k \neq l$;
    \item For each $k \in \{1,\ldots,r-1\}$, either $i_k < i$ or $i_k > j$;
    \item The induced subgraph of G on the vertex set $\{i_0,\ldots,i_r\}$ has no induced cycle.
\end{itemize}
}
\end{definition}

\begin{remark}\label{rem:grob}{\rm
Let $G$ be a graph on $[n]$. Corresponding to an admissible path $\pi: i=i_{0},i_{1},\ldots,i_{r}=j$ in $G$, we associate the monomial
$$ u_{\pi}=\bigg(\prod_{i_{k}>j} x_{i_{k}}\bigg)\bigg(\prod_{i_{l}<i} y_{i_{l}}\bigg).$$
Let $\prec$ be the lexicographic order on $R=K[x_1,\ldots,x_n,y_1,\ldots,y_n]$ induced by $x_1\succ x_2\succ \cdots \succ x_n \succ y_1\succ y_2\succ \cdots\succ y_n$. Then $\mathcal{G}=\{u_{\pi}f_{ij}\mid \pi \text{ is an admissible path from } i \text{ to } j\}$ is a reduced Gr\"{o}bner basis of $J_{G}$ with respect to $\prec$ by \cite[Theorem 2.1]{hhhkr10}. Therefore, 
$$\mathscr{G}(\mathrm{in}_{\prec}(J_{G}))=\{u_{\pi}x_{i}y_{j}\mid \pi\,\, \text{is an admissible path from}\,\, i\,\, \text{to}\,\, j\,\, \text{with}\,\, i<j\},$$
where $\mathscr{G}(I)$ denotes the minimal monomial generating set of a monomial ideal $I$.
}
\end{remark}

\begin{lemma}\label{reg-lowerbnd}
Let $G$ be a non-empty graph on $p$ vertices. Then $\reg(R/J_{W(G)})\geq p+1$.
\end{lemma}
\begin{proof}
    Given that $G$ is a non-empty graph on $p$ vertices. If $p=2$, then $G=K_2$ and $W(G)$ is a path graph of length $3$. In this case, we have $\reg(R/J_{W(G)})=3$ by \cite[Theorem 1.1]{mm13}. Thus, we may assume $p>2$. Label the vertices of $G$ from $3$ to $p+2$ such that $\{3,4\}$ is an edge of $G$. Now we will label the vertices of $W(G)$ in a special manner: label the pendant vertex attached to $3$ as $1$, $4$ as $2$, $5$ as $p+3$, $6$ as $p+4$, \ldots, $p+2$ as $2p$. Now, consider the initial ideal $\ini_{\prec}(J_{W(G)})$ of $J_{W(G)}$ with respect to the lexicographic order $\prec$ mentioned in \Cref{rem:grob}. Let $\mathcal{H}$ be the hypergraph such that $I_{\mathcal{H}}=\ini_{\prec}(J_{W(G)})$. Due to our choice of labeling of vertices and the description of minimal generators of $\ini_{\prec}(J_{W(G)})$ given in \Cref{rem:grob}, one can see that
    $\{x_1,y_2,x_3,x_4\},\{x_5,y_{p+3}\},\{x_6,y_{p+4}\},\ldots,\{x_{p+2},y_{2p}\}\in E(\mathcal{H})$. Let us take 
    $$M=\{ \{x_1,y_2,x_3,x_4\},\{x_5,y_{p+3}\},\{x_6,y_{p+4}\},\ldots,\{x_{p+2},y_{2p}\}\}.$$
    \textbf{Claim:} $M$ forms an induced matching in $\mathcal{H}$.\\
    \textit{Proof of the claim.} Observe that for any two distinct edges $e_i,e_j \in M$, we have $e_i\cap e_j=\emptyset$. Therefore, edges in $M$ are pairwise disjoint. Let 
    $$K=\bigcup_{e\in M}e=\{x_1,x_3,x_4,x_5,x_6,\ldots,x_{p+2},y_2,y_{p+3},y_{p+4},\ldots,y_{2p}\}.$$ 
    We have to show that the only edges of $\mathcal{H}$ contained in $K$ are elements of $M$. Suppose there exists $e\in E(\mathcal{H})\setminus M$ such that $e \subset K$. Note that $\{x_{i},y_2\} \notin E(\mathcal{H})$ as the only neighbor of $2$ in $W(G)$ is $4$ and $4>2$. Also, for $1 \leq l \leq p-2$, $\{x_i,y_{p+2+l}\} \notin E(\mathcal{H})$ if $i \neq 4+l$ as the only neighbor of $p+2+l$ is $4+l$. Hence $\vert e \vert > 2$. Now, corresponding to $e$ there exists an admissible path $\pi: i = i_{0},i_{1},\ldots,i_{r} = j $ with $r \geq 2$ in $W(G)$ such that $\supp(u_{\pi}x_iy_j)=e$. Observe that if $j=i_r=2$, then the only possibility is $e=\{x_1,y_2,x_3,x_4\}$, which is in $M$ and gives a contradiction. Therefore, $j=i_r \in \{p+3,p+4,\ldots,2p\}$ as $y_j \in K$. Let $j=i_r=p+2+l$ for some $1 \leq l \leq p-2$. Since $p+2+l$ is a pendant vertex of $W(G)$ adjacent to $4+l$, the only choice for $i_{r-1}$ is $4+l$. Since we assumed $p>2$, it follows from the definition of admissible path that $i_{r-1} =4+l< i$, and so, $y_{4+l} \in e$. This gives a contradiction to the fact that $e \subset K$ as $y_{4+l}\notin K$ for any $1\leq l\leq p-2$. Hence, our assumption is wrong and $M$ forms an induced matching in $\mathcal{H}$.\par 
    Now, by the above claim and \Cref{thrm:9}, we get
    \begin{align*}
        \reg(R/\mathrm{in}_{\prec}(J_{W(G)})) & = \reg(R/I_{\mathcal{H}}) \\
        & \geq \sum_{e \in M}(\vert e \vert -1 ) \\
        & =  2(p-2)+4-(p-1) \\
        & =  p+1
    \end{align*}
Hence, by \cite[Corollary 2.7]{cv20}, $\reg(R/J_{W(G)})=\reg(R/\mathrm{in}_{\prec}(J_{W(G)}))\geq p+1$.
\end{proof}

\begin{theorem}\label{thm:reg-bound-gapfree}
    Let $G$ be a gap-free graph. Then 
    $$\reg(R/J_{W(G)})=\vert V(G)\vert+1.$$
\end{theorem}
\begin{proof} Since the graph $G$ is gap-free, $\im(G)=1$. Hence, by \Cref{thm:reg-bound-gen} and \Cref{reg-lowerbnd}, we have $\reg(R/J_{W(G)})= \vert V(G)\vert+1$. 
    \end{proof}

\section{Cohen-Macaulay binomial edge ideals of graphs in $\G_2$}
In this section, we classify all graphs having Cohen-Macaulay binomial edge ideals that belong to the class $\G_2$. In particular, it gives a new construction of Cohen-Macaulay binomial edge ideals arising from the generalized corona product of graphs. We begin by recalling some fundamental results concerning the Cohen–Macaulay property of binomial edge ideals.

\begin{lemma}[{\cite[Lemma 2.5]{rr14}}]\label{lem:5} Let $G$ be a connected graph. Then $J_{G}$ is unmixed if and only if $c_G(T)= \vert T \vert +1$ for all $T \in \mathcal{C}(G)$.
\end{lemma}

The \textit{cone} on a graph $H$, denoted by $\cone(v,H)$, is a graph with vertex set $V(H)\sqcup \{v\}$ and edge set $E(H)\cup\{\{u,v\}\mid u\in V(H)\}$.

\begin{lemma}[{\cite[Lemma 3.4]{rr14}}] \label{thrm:6} Let $H=\sqcup_{i=1}^{r}H_i$ be a graph with $H_{i}$ connected components with $r \geq 1$ and let $G=\cone(v,H)$. If $J_G$ is unmixed, then $H$ has at most two connected components. 
\end{lemma}
\begin{theorem}[{\cite[Proposition 3.2]{ss24}}]  \label{thrm:7} 
    Let $G$ be a graph and $v \in V(G)$. If $J_G$ is Cohen-Macaulay, then $J_{G_{v}}$ is Cohen-Macaulay.
\end{theorem}
\begin{theorem} [{\cite[Theorem 4.8]{bms22}}] \label{thrm:8}
    Let $H = H_1 \sqcup H_2$ with $H_1, H_2$ connected and $G= \cone(v,H)$. Then $J_{H_{1}}$ and $J_{H_{2}}$ are Cohen-Macaulay if and only if $J_G$ is Cohen-Macaulay.
\end{theorem}
Before classifying the graphs with Cohen-Macaulay binomial edge ideals in the class $\G_2$, we have to establish a lemma regarding the Krull dimension of binomial edge ideals of graphs belong to a subclass of $\G_2$. Let $\G_2'$ be a subclass of $\G_2$ containing the graphs of the form $D=K_{p}\circ_{S}(H_1,\ldots,H_{|S|})$ for some $p$. The following lemma gives the formula of $\dim(R/J_{D})$ for any $D\in \G_2'$.

\begin{lemma}\label{lem:dim}
    Let $D=K_{p}\circ_{S}(H_1,\ldots,H_{|S|})$ be any graph belong to the class $\G_2'$. Then 
    $$\dim(R/J_D)=p-|S|+1+\sum_{i=1}^{|S|}\dim(R_{i}/J_{H_i}).$$
\end{lemma}
\begin{proof}
    We proceed by induction on $|S|$. If $|S|=0$, then $D$ is a complete graph on $p$ vertices, and thus, $\dim(R/J_D)=p+1$. Let us consider $|S|\geq 1$ and $S=\{v_1,\ldots,v_{\ell}\}$ such that $v_i$ is adjacent to all vertices of $H_i$ for all $1\leq i\leq \ell$. For simplicity of notation, let us write $v=v_{\ell}$. Then we see that $D-v=H_{\ell}\sqcup D'$ with $D'\in \G_2'$ and $D_v,D_v-v\in\G_2'$. Again note that exactly $(\ell-1)$ graphs $H_1,\ldots,H_{\ell-1}$ are attached with the base graphs in $D'$, $D_v$ and $D_v-v$. Therefore, using the induction hypothesis, we have
    \begin{align*}
    \dim(R_v/J_{D-v}) & = \dim(R_{\ell}/J_{H_{\ell}})+(p-1)-(\ell-1)+1+\sum_{i=1}^{\ell-1}\dim(R_i/J_{H_i})\\
     & = p-\ell+1 + \sum_{i=1}^{\ell}\dim(R_i/J_{H_i})\\
    \dim(R/J_{D_{v}})& = (p+|V(H_{\ell})|)-(\ell-1)+1 + \sum_{i=1}^{\ell-1}\dim(R_i/J_{H_i}) \\
    & \leq \dim(R_v/J_{D-v}) \quad (\text{since } \dim(R_{\ell}/J_{H_{\ell}})\geq |V(H_{\ell})|+1)\\
    \dim(R_v/J_{D_{v}-v)}) & = (p+|V(H_{\ell})|-1)-(\ell-1)+1 + \sum_{i=1}^{\ell-1}\dim(R_i/J_{H_i})\\
    & < \dim(R_v/J_{D-v}) \quad (\text{since } \dim(R_{\ell}/J_{H_{\ell}})\geq |V(H_{\ell})|+1).
\end{align*}
Now, using the short exact sequence (\ref{shrt-exact-1}), \Cref{lem:6}, and the above inequalities, we get
\begin{align*}
    \dim(R_v/J_{D-v} \oplus R/J_{D_v}) & = \mathrm{max}\{\dim(R/J_D), \dim(R_v/J_{D_{v}-v})\} \\
  \implies\hspace{0.5cm} \dim(R_v/J_{D-v})  & = \dim(R/J_D)
\end{align*}
Hence, we have $\dim(R/J_D)=p-\ell+1+\sum_{i=1}^{\ell}\dim(R_i/J_{H_i})$.
\end{proof}

\begin{theorem}\label{thrm:cohen-macly}
 Let $D=G\circ_{S}(H_1,\ldots,H_{\vert S\vert})$ be any connected graph in $\G_2$ with $G$ non-empty. Then the following are equivalent:
    \begin{enumerate}
        \item $J_{D}$ is Cohen-Macaulay;
        \item $G$ is complete, each $H_i$ is connected with Cohen-Macaulay $J_{H_i}$, and whenever $\vert S\vert=\vert V(G)\vert$ at least one of $H_{i}$'s is complete.
        \end{enumerate}
\end{theorem}

\begin{proof} (1) $\implies$ (2): Suppose $J_{D}$ is Cohen-Macaulay. Assume by contradiction, $G$ is not complete. Then there exists a pair of vertices $u$ and $v$ in $G$ such that $u$ and $v$ are not adjacent. Consider the set of vertices $T=N_{G}(u)\setminus A_{G}$. Then $T\subset B_{G}\subset S$, and from the construction of $D$, we have $T\in \mathcal{C}(D)$. Since $D$ is connected, $G$ is also connected, and thus, $u$ and $v$ are connected through some paths in $G$. Observe that neighbors of $u$ in $G\setminus T$ are either free vertices or there is no neighbor of $u$ in $G\setminus T$. Consequently, in $G\setminus T$, $u$ and $v$ belong to two different connected components as there is no edge between $u$ and $v$. Therefore, $c(D\setminus T)\geq |T|+2$ as each vertex of $T$ is a cut vertex in $D$, which is a contradiction to the fact that $J_{D}$ is unmixed by \Cref{lem:5}. Hence, our assumption was wrong, and $G$ is a complete graph. Next, if $|S|=\emptyset$, then $D=G$ is a complete graph and there is nothing to prove. Note that $|V(G)|\geq 2$ as $G$ is non-empty. Now, assume $S=\{v_1,\ldots,v_{\ell}\}$ for some $\ell\geq 1$, and consider the graph $D'=D_{v_{1}\cdots\widehat{v_i}\cdots v_{\ell}}$. Then, by \Cref{thrm:7}, $J_{D'}$ is Cohen-Macaulay. Since $G$ is connected and all the vertices of $V(G)\setminus \{v_i\}$ are free vertices in $D'$, it follows that $v_i$ is adjacent to all other vertices in $D'$. Thus, we can see $D'=\cone(v_i,H)$ where $H= H' \sqcup H_i$ for some complete graph $H'$. Then, $H_i$ is connected by \Cref{thrm:6}, and $J_{H_i}$ is Cohen-Macaulay by \Cref{thrm:8}. Since $v_i$ has been chosen arbitrarily, all $H_i$'s are connected with Cohen-Macaulay $J_{H_i}$. Now, consider the case $\vert S\vert=\vert V(G)\vert = p$ (say). We have to show that at least one of the $H_i$'s is complete. Assume, to the contrary, that none of the $H_i$ is complete. Let us choose $T_i \in \mathcal{C}(H_i)\setminus \{\emptyset\}$ for each $1 \leq i \leq p$, and consider the set $T=V(G) \cup T_1 \cup\cdots\cup T_p$. Pick any $s \in T$. If $s \in T_i$ for some $i$, then $s$ is a cut vertex of the graph $D- (T\setminus\{s\})$ as $s$ is a cut vertex of the graph $H_{i}-(T_i\setminus \{s\})$. Suppose $s \in V(G)$ and the graph $H_s$ is attached to $s$ in $D$ for some $1\leq s\leq p$. Since $H_s$ is connected with unmixed $J_{H_s}$, $s$ is adjacent to $\vert T_s \vert +1$ connected components of $ H_s \setminus T_s$ by \Cref{lem:5}. $T_s$ being a non-empty cutset of $H_s$, we have $|T_s|+1\geq 2$, which implies $s$ is a cut vertex in $D-(T\setminus \{s\})$. Therefore, $T$ is a cutset of $D$. Moreover, we have
\begin{align*}
    c_{D}(T) & = c_{H_1}(T_1)+\cdots+c_{H_p}(T_p) \\
    & = (\vert T_1 \vert +1)+ \cdots+ (\vert T_p \vert +1) \quad (\text{by \Cref{lem:5}})\\
    & = \sum_{i=1}^{p} \vert T_i \vert + p \\
    & = \vert T \vert.
\end{align*}
By \Cref{lem:5}, $c_{D}(T)=|T|$ is a contradiction to the fact that $J_D$ is unmixed. Hence, at least one of the $H_i$'s should be complete whenever $|S|=|V(G)|$.\par

(2) $\implies$ (1): Let us assume $D$ satisfies the condition (2) of the statement of the theorem. We have to prove that $J_{D}$ is Cohen-Macaulay. Since all $H_i$'s are connected with Cohen-Macaulay binomial edge ideals, it follows from our \Cref{thm:depth-equal} that $\depth(R/J_D)= \vert V(D) \vert + 1$. Again, note that $\dim(R_i/J_{H_i})=|V(H_i)|+1$ as each $H_i$ is connected with Cohen-Macaulay binomial edge ideal. Due to the assumption of condition (2), we can easily see that $D$ belongs to the class $\G_2'$ considered in \Cref{lem:dim}. Therefore, using the formula for dimension established in \Cref{lem:dim}, we have $\dim(R/J_D)=|V(D)|+1$. Hence, $J_D$ is Cohen-Macaulay, and this completes the proof.
\end{proof}

\subsection*{Acknowledgements} The third author is partially supported by the Anusandhan National Research Foundation (ANRF), Government of India, under the ARG-MATRICS Grant No. ANRF/ARGM/2025/002203/MTR.

\bibliographystyle{abbrv}
	\bibliography{ref}

\end{document}